\documentclass[preprint,11pt]{elsarticle}

\usepackage{amssymb}
\usepackage{graphicx}
\usepackage{amsmath}
 \usepackage{amsthm}
\newtheorem{theorem}{Theorem}
\newtheorem{lemma}{Lemma}
\usepackage{xcolor}

\newtheorem{example}{Test problem}
\theoremstyle{case}

\usepackage{algorithm}
\usepackage{algpseudocode}
\usepackage{graphicx}
\usepackage{subcaption}
\usepackage{caption}
\newtheorem{definition}{Definition}

\begin{document}

\begin{frontmatter}

\title{\textbf{A Robust ADI-Based Direct Discontinuous Galerkin Method for Two-Dimensional Singularly Perturbed Parabolic  PDEs}} 

	\author[mymainadress]{Vijay~Kumar}
\ead{vijaykumarmathematics@gmail.com}

\author[mymainadress]{Gautam~Singh\corref{mycorrespondingauthor}}
\cortext[mycorrespondingauthor]{CorrespondingAuthor}
\ead{gautam@nitt.edu}

\address[mymainadress]{{Department of Mathematics, National
		Institute of Technology, Tiruchirappalli},
	{Tiruchirappalli},
	{620015},
	{Tamil Nadu},
	{India}.}

\begin{abstract}
This study proposes two-dimensional one-parameter singularly perturbed parabolic problems and develops an operator-splitting direct discontinuous Galerkin (DDG) method for their numerical solution. To improve computational efficiency, an alternating direction implicit (ADI) splitting approach decomposes the original two-dimensional problem into  two separate one-dimensional subproblems. The spatial discretization is performed using the DDG method on a piecewise-uniform Shishkin mesh, enabling accurate resolution of the sharp boundary layers that arise in singularly perturbed problems (SPPs). For the temporal discretization is carried out with the backward Euler scheme to ensure robustness and stability. The theoretical analysis establishes an optimal convergence rate of order $k$ in the spatial discretization, where $k$ denotes the degree of the polynomial basis functions, and $(\Delta t)$-order convergence in temporal dimensions. We prove that the resulting fully discrete scheme is consistent, stable, and convergent. Numerical experiments confirm the theoretical results and demonstrate the effectiveness and robustness of the proposed method.
\end{abstract}


\end{frontmatter}
\section{Introduction}\label{secint}
SPPs  are found in many scientific and engineering fields, such as control theory~\cite{kokotovic}, transport and flow modeling~\cite{Modeling_Fluid}, and fluid mechanics~\cite{Schlicting}. A defining feature of these problems is the presence of small perturbation parameters associated with the highest-order derivatives. This leads to the formation of boundary or interior layers, within which the solution undergoes sharp spatial variations. As a result, SPPs are well known for their analytical complexity and computational difficulty.

Standard numerical methods, including the finite difference method (FDM) and finite element method (FEM), may fail to resolve the steep gradients associated with boundary and interior layers, particularly when the perturbation parameter is small. In convection-dominated regimes, conventional discretizations may also produce non-physical oscillations, thereby compromising the accuracy and reliability of the numerical solution; see, for instance,~\cite{Roos}. These difficulties motivate the development of numerical methods that remain stable and accurate uniformly with respect to the perturbation parameter. A widely adopted approach is the use of layer-adapted meshes, which allocate computational points strategically in regions where sharp variations occur~\cite{Lins}. Among the commonly employed layer-adapted meshes, the Shishkin mesh uses a piecewise-uniform distribution with fine and coarse regions, whereas the Bakhvalov mesh uses a smooth grading to concentrate mesh points near layers; see, e.g.,~\cite{Roos}. A variety of numerical schemes have subsequently been developed on such meshes to achieve parameter-uniform convergence. Comprehensive treatments of these approaches, including FDM and FEM, can be found in~\cite{Lins,Roos}.

Several operator-splitting strategies have been developed in the literature for the numerical treatment of two-dimensional parabolic SPPs within the finite difference framework. Clavero et al. \cite{clavero_bound} developed a numerical scheme that uses a fractional step method for time integration and finite difference spatial discretization on a specially designed mesh. Das and Natesan \cite{Das_$2$D_parabolic} created an FDM that uses a fractional-step technique to approximate the temporal derivative. This method converges to the second order in time. Moreover, Majumdar and Natesan \cite{Majumdar_$2$D_parabolic} proposed an ADI FDM that utilizes a uniform temporal discretization and an upwind scheme on a Shishkin mesh to effectively resolve boundary layers in the spatial domain.

Operator-splitting techniques within the FEM have been extensively examined for the efficient resolution of high-dimensional and time-dependent problems. Ganesan \cite{Ganesan_SUPG} proposed a heterogeneous FEM for population balance equations by partitioning the original problem into lower-dimensional subproblems across physical and internal coordinates. This approach simplifies the formulation and enables the use of appropriate discretization techniques, like  streamline upwind Petrov–Galerkin (SUPG) methods, to keep oscillations in check in advection-dominated regimes. Ganesan and Tobiska \cite{Ganesan_SUPG2} also applied these concepts to high-dimensional parabolic problems using quadrature- and nodal-based splitting methods, which led to stable and optimally convergent solutions. Das and Natesan \cite{Parabolic_FEM_ADI} proposed an ADI-type operator splitting streamline-diffusion FEM.  Moreover,  Ganesan and Singh \cite{Ganesan_SUPG3} made a related contribution by creating an operator-splitting scheme for the radiative transfer equation. This scheme combines SUPG and discontinuous Galerkin (DG) methods for spatial–angular discretization with backward Euler time integration. 

Ahmed et al. \cite{Ahmed1} used operator splitting on population balance equations in a similar way, breaking the problem down into a pure advection and a convection–diffusion subproblem. They used the backward Euler method to break up the time, and  DG schemes to make it stable. Subsequent enhancements in \cite{Ahmed2} integrated SUPG and DG methodologies within the splitting framework, yielding comparative analyses of stabilization techniques for convection-dominated issues.

Several operator-splitting strategies have been developed in the literature for the numerical treatment of two-dimensional parabolic SPPs within the DG methods, including weak Galerkin FEM (WG-FEM). Kumar et al. \cite{Naresh_one_para} proposed an operator-splitting ADI scheme within the WG-FEM framework for semilinear parabolic SPPs. The original two-dimensional problem is decomposed into a sequence of one-dimensional subproblems, which are solved using WG-FEM in the spatial direction, while temporal discretization is carried out using the Crank–Nicolson scheme on a Bakhvalov mesh. In a subsequent work \cite{Naresh_two_para}, extended this approach by combining the  operator-splitting   scheme with a high-order WG-FEM to address two-parameter singularly perturbed parabolic  problems. The resulting scheme employs a similar dimensional decomposition and utilizes a nonuniform Bakhvalov mesh, leading to improved accuracy and efficient resolution of boundary layers.

Similarly, Raina et al. \cite{Raina_one-para} developed an operator-splitting ADI scheme within the WG-FEM for two-dimensional time-dependent SPPs by decomposing the original problem into one-dimensional subproblems, thereby reducing computational complexity and storage requirements. The method employs WG-FEM for spatial discretization and the backward Euler scheme for time integration. In a subsequent study \cite{Raina_two-para}, an ADI operator-splitting approach was combined with the WG-FEM to solve two-parameter singularly perturbed parabolic problems on layer-adapted meshes, leading to improved computational efficiency while maintaining stability.


The DDG method provides an effective framework for the discretization of diffusion operators by incorporating suitable numerical fluxes across element interfaces. This formulation retains the local and flexible structure of DG methods while offering high-order accuracy and stability. An important advantage of the DDG framework is its compatibility with non-uniform and layer-adapted meshes, such as Shishkin, Bakhvalov, and exponentially graded meshes, which are particularly effective in resolving the sharp gradients arising in SPPs. To the best of our knowledge, the DDG method has not yet been systematically investigated for two-dimensional singularly perturbed parabolic problems. More specifically, an ADI-type operator-splitting strategy based on the DDG framework for two-dimensional one-parameter singularly perturbed parabolic problems has yet to be investigated. The proposed strategy transforms the original two-dimensional problem into a sequence of one-dimensional subproblems, thereby reducing the number of degrees of freedom and leading to significantly smaller algebraic systems. Moreover, the DDG framework is well suited for problems with limited regularity and is capable of accurately resolving sharp boundary layers while maintaining stability and high-order accuracy. Motivated by these advantages, we develop an ADI-type operator-splitting DDG method for the model problem \eqref{E1}, corresponding to the one-parameter cases. The proposed method combines the computational efficiency of operator splitting with the flexibility and accuracy of the DDG framework, and rigorous stability and convergence analyses are established.

The remainder of the article is organized as follows. Section~\ref{sec(2.1)} introduces the model problem together with its operator-splitting formulation into two one-dimensional subproblems. The solution decomposition, spatial discretization, and the basic notation of the DDG framework are presented in Section~\ref{sec(2.2)}. Section~\ref{sec(2.3)} introduces the DDG formulation, establishes the permissibility conditions of the DDG numerical fluxes, and presents the fundamental theoretical results, including  coercivity, and stability. The error analysis and numerical experiments are presented in Sections~\ref{sec(2.4)} and~\ref{sec(2.5)}, respectively. Finally, Section~\ref{Conclusion} concludes the paper.

Throughout this paper, $C$ denotes a generic positive constant whose value may vary from line to line but which is independent of the perturbation parameters $\mu$ and the mesh parameter $\mathcal{\mathcal{N}}$ and $M$.

\section{The continuous problem}\label{sec(2.1)}
Consider the two-dimensional one-parameter singularly perturbed parabolic initial boundary-value problem (IBVP) defined on the domain $\mathcal{D}$ = $I$ $\times$ $(0, T]$, with $I = I_x\times I_y$, with $I_x=(0,1)$ and $I_y=(0,1)$:
\begin{equation}\label{E1}
	\left\{
	\begin{array}{ll}
		\frac{\partial u}{\partial t}-\mu \Delta u + \boldsymbol{b}(x,y)\cdot\nabla u + a(x,y) u = r(x,y,t), & (x,y,t) \in \mathcal{D}, \\[8pt]
		
		u(x,y,0)=u_0(x,y) & (x,y) \in \overline{I}, \\[8pt]
		
		u(x,y,t)=0 & (x,y,t) \in \partial I\times(0,T].
	\end{array}\right.
\end{equation}
Here, $0 < \mu \ll 1$ denotes a small perturbation parameter, and $T$ is a positive constant. The advection vector $\boldsymbol{b}(x,y) = (b_1(x,y), b_2(x,y))$ satisfies $b_1(x,y) \ge \beta_x > 0$ and $b_2(x,y) \ge \beta_y > 0$, while the reaction coefficient $a(x,y)=a_{1}(x,y)+a_{2}(x,y)$ satisfies $a_{1}(x,y),a_{2}(x,y)\geq 0$, $\forall (x,y) \in \overline{I}$. Furthermore, we assume that $r$ and $u_0$ are sufficiently smooth to ensure the existence and uniqueness of the solution to problem \eqref{E1}. When the perturbation parameter $\mu$ is very small, the solution $u(x,y,t)$ develops sharp boundary layers near $x = 1$ and $y = 1.$ 


\subsection{ADI operator splitting and problem decomposition}\label{sec(2.1.1)}
We introduce the following linear operators:
\begin{equation*}\label{E2}
	\left\{
	\begin{array} {ll}
		\mathcal{L}_{x,\mu}=-\mu \frac{\partial^{2}}{\partial x^{2}} + b_{1}(x,y)\frac{\partial}{\partial x}  + a_1(x,y),  \forall y \in I_y, \\[8pt]
			\mathcal{L}_{y,\mu}=-\mu \frac{\partial^{2}}{\partial y^{2}} + b_{2}(x,y)\frac{\partial}{\partial y}  + a_2(x,y),  \forall x \in I_x, 
	\end{array}\right.
\end{equation*}
where $	\mathcal{L}_{x,\mu}$ and $		\mathcal{L}_{y,\mu}$ are linear operators in the  $x$- and $y$-directions, respectively. In addition, we also assume that $a_{1} - \frac{1}{2}\frac{\partial b_1}{\partial x}\geq\gamma>0$ and $a_{2} - \frac{1}{2}\frac{\partial b_2}{\partial y}\geq\gamma>0$, for some constant $\gamma$. We further suppose that the source term $r(x,y,t)$ admits a splitting of the form $r= r_{1} + r_{2}$ such that it satisfies the following  boundary compatibility conditions:
\[r_{1}(x,0,t)=r_{2}(0,y,t)=r_{1}(x,1,t)=r_{2}(1,y,t)=0.\]
Under this assumption, problem \eqref{E1} can be equivalently expressed using the directional operators $	\mathcal{L}_{x,\mu}$ and $	\mathcal{L}_{y,\mu}$ as
\[	\frac{\partial u}{\partial t} +(	\mathcal{L}_{x,\mu} +\mathcal{L}_{y,\mu})u = r_{1}(x,y,t)+ r_{2}(x,y,t) \quad \text{on} \quad I \times (0,T].\]
Next, we decomposed the $2$D model problem~\eqref{E1} into two separate $1$D subproblems along the $x$- and $y$-directions, respectively. Thus, we use two $1$D DDG spaces. Time discretization is performed using the backward Euler method. In the temporal discretization, the mesh points on the time domain [0, $T$] are given by
\begin{eqnarray}
	T^{M}=\{t_{m}=m\Delta t,\quad m=0,1,\ldots,M,\quad \Delta t= T/M\}.\nonumber
\end{eqnarray}
The length of equidistant mesh intervals is denoted by $\Delta t=(t_{m}-t_{m-1})$. Using this approach, we now solve the following subproblems on each time subinterval $(t_{m-1}, t_m]$, for $m = 1, \ldots, M$:
\vskip 2mm
\textbf{Subproblem1}
\begin{equation}\label{E3}
	\left\{
	\begin{array} {ll}
		\text{Find} \quad v : I \times (t_{m-1}, t_{m}] \to \mathbb{R} \quad \text{such that} \quad  \forall y \in I_y , \\[8pt] 
		\frac{\partial v}{\partial t} +\mathcal{L}_{x,\mu}v = r_{1}(x,y,t_{m}),  \quad in \quad I \times (t_{m-1}, t_{m}] , \\[8pt]
		v(x,y,t)=0, \quad  (x,y,t) \in \partial I_{x}\times I_{y}\times(t_{m-1}, t_{m}],\\[8pt]
		v(x,y,t_{m-1})=u(x,y,t_{m-1}), \quad (x,y) \in I.
	\end{array}\right.
\end{equation}
In this step, for each fixed $y \in I_y$, the corresponding subproblem is subsequently solved along the $x$-direction to obtain the intermediate solution $v(x,y,t_{m})$. This intermediate solution serves as the initial condition for the next subproblem.
\vskip 2mm
\textbf{Subproblem2}
\begin{equation}\label{E4}
	\left\{
	\begin{array} {ll}
		\text{Find} \quad w : I \times (t_{m-1}, t_{m}] \to \mathbb{R} \quad \text{such that} \quad  \forall x \in I_x , \\[8pt] 
		\frac{\partial w}{\partial t} +\mathcal{L}_{y,\mu}w = r_{2}(x,y,t_{m}),  \quad in \quad I \times (t_{m-1}, t_{m}] , \\[8pt]
		w(x,y,t)=0, \quad  (x,y,t) \in  I_{x}\times \partial I_{y}\times(t_{m-1}, t_{m}],\\[8pt]
		w(x,y,t_{m-1})=v(x,y,t_{m}), \quad (x,y) \in I.
	\end{array}\right.
\end{equation}
In this step, for each fixed $x \in I_x$, the corresponding subproblem is subsequently solved along the $y$-direction. The coupling between Subproblem~\eqref{E3} and Subproblem~\eqref{E4} is established through the initial condition, where the intermediate solution $w(x,y,t_{m})$ is used to advance the solution to the next time level.

\section{A priori bounds and domain discretization}\label{sec(2.2)}
In this section, we derive estimates for the smooth and boundary layer components of the solution, along with their corresponding derivatives. We also introduce the layer-adapted mesh employed in our analysis. Finally, we summarize the basic notation used throughout the paper.

\subsection{A priori bounds for the continuous solution}\label{sec(2.2.1)}
In this subsection, we present a lemma that provides bounds for the solution of \eqref{E1} and its derivatives. The exact solution $u(x,y,t)$ of the IBVP \eqref{E1} is decomposed into a smooth component $S$ and a singular component $E$ as
\[u = S + E.\]
Assuming sufficient regularity of the solutions $v$ of \eqref{E3} and $w$ of \eqref{E4}, we derive estimates for these functions for all integers $0 \leq i \leq k+2$ and $0 \leq i + j \leq k+2$. Then, following the technique in \cite{boud_solution}, we obtain
\begin{eqnarray}
	&&\left|\frac{\partial^{i + j}v}{\partial x^{i}\partial t^{j}}\right|  
	\leq C\left(1+\mu^{-i}\exp\left(\frac{-\beta_{x}(1-x)}{\mu}\right)\right), 
	\quad \forall y \in I_{y}, \nonumber\\
	&&\left|\frac{\partial^{i + j}w}{\partial y^{i}\partial t^{j}}\right|  
	\leq C\left(1+\mu^{-i}\exp\left(\frac{-\beta_{y}(1-y)}{\mu}\right)\right), 
	\quad \forall x \in I_{x}. \nonumber
\end{eqnarray}
\begin{lemma}\label{lemma1}
	Assume that the functions $r_{1}, r_{2}, \beta_x,$ and $\beta_y$ are sufficiently smooth. Then, the solutions $v$ and $w$ admit the decompositions  $v = S^{*} + E^{*}, \quad w = S^\dagger + E^\dagger$,	where $S^{*}, S^\dagger$ denote the smooth components and $E^{*}, E^\dagger$ denote the layer components. Then, $\forall$ integers $0 \leq i \leq k+2$ and $0 \leq i + j \leq k+2$, the following bounds hold:
	\begin{eqnarray}
		&&\left\{
		\begin{array}{ll}
			\left|\dfrac{\partial^{i+j}S^{*}}{\partial x^{i}\partial t^{j}}\right|  
			\leq C\left(1+\mu^{k+1-i}\right), \\[6pt]
			\left|\dfrac{\partial^{i+j}E^{*}}{\partial x^{i}\partial t^{j}}\right|  
			\leq C\left(\mu^{-i}\exp\left(\dfrac{-\beta_{x}(1-x)}{\mu}\right)\right), 
			\quad \forall y \in I_{y},
		\end{array}
		\right. \nonumber \\[10pt]
		&&\left\{
		\begin{array}{ll}
			\left|\dfrac{\partial^{i+j}S^\dagger}{\partial y^{i}\partial t^{j}}\right|  
			\leq C\left(1+\mu^{k+1-i}\right), \\[6pt]
			\left|\dfrac{\partial^{i+j}E^\dagger}{\partial y^{i}\partial t^{j}}\right|  
			\leq C\left(\mu^{-i}\exp\left(\dfrac{-\beta_{y}(1-y)}{\mu}\right)\right), 
			\quad \forall x \in I_{x}.
		\end{array}
		\right. \nonumber
	\end{eqnarray}
\end{lemma}
\begin{proof}
	The proof follows from the technique presented in \cite{clavero_bound} and \cite{gowrisankar_bound}.
\end{proof}
\subsection{Spatial domain discretization}\label{sec(2.2.2)}
The spatial domain is discretized using a piecewise-uniform Shishkin mesh. Let $I = I_x \times I_y$, where $I_x = (0,1)$ and $I_y = (0,1)$. The resulting two-dimensional mesh is constructed as the tensor product of two one-dimensional Shishkin meshes. Since the problem \eqref{E1} exhibits boundary layers near $x = 1$ and $y = 1$, we employ piecewise-uniform meshes on both $I_x$ and $I_y$. For simplicity, we use the same number $\mathcal{N}$ of subintervals in both domains $I_x$ and $I_y$. The mesh transition parameters are defined by
\[
\sigma_{j} = \min\left\{\frac{1}{2}, \frac{(k+1)}{\beta_j}\mu \ln \mathcal{\mathcal{N}} \right\}, \quad j \in \{x,y\}.
\]
In practice, we take $\sigma_{j} = \frac{(k+1)}{\beta_j}\mu \ln \mathcal{\mathcal{N}}$ whenever $\frac{(k+1)}{\beta_j}\mu \ln \mathcal{\mathcal{N}} \leq \tfrac{1}{2}$; otherwise, the mesh reduces to a uniform mesh. The intervals are decomposed as $I_x = I_{x1} \cup I_{x2}, \quad 
I_y = I_{y1} \cup I_{y2},$
where
\[
I_{x1} = [0, 1-\sigma_x], \quad I_{x2} = [1-\sigma_x, 1], \quad
I_{y1} = [0, 1-\sigma_y], \quad I_{y2} = [1-\sigma_y, 1].
\] Each subinterval is uniformly partitioned into $\mathcal{N}/2$ mesh intervals. Accordingly, the mesh points in the $x$-direction is given by $I_x = \{0 = x_0 < x_1 < \cdots < x_{\mathcal{N}/2} = 1-\sigma_x < x_{\mathcal{N}/2+1} < \cdots < x_{\mathcal{N}} = 1\},$ and the corresponding mesh elements are defined by $\mathcal{K}_{x_{i}} = (x_{i-1}, x_{i})$ for $i = 1,\ldots,\mathcal{N}$. Similarly, in the $y$-direction, $I_y = \{0 = y_0 < y_1 < \cdots < y_{\mathcal{N}/2} = 1-\sigma_y < y_{\mathcal{N}/2+1} < \cdots < y_{\mathcal{N}} = 1\},$ with mesh elements $\mathcal{K}_{y_{j}} = (y_{j-1}, y_{j})$ for $j = 1,\ldots,\mathcal{N}$. The mesh step sizes in the $x$- and $y$-directions are defined as $h_i = x_{i} - x_{i-1}, \quad h_j = y_{j} - y_{j-1},$  respectively.

In the $x$-direction, the mesh sizes satisfy
\[
h_i =
\begin{cases}
	\dfrac{2(1-\sigma_x)}{\mathcal{N}}, & i = 1, 2, \ldots, \dfrac{\mathcal{N}}{2}, \\[6pt]
	\dfrac{2\sigma_x}{\mathcal{N}}, & i= \dfrac{\mathcal{N}}{2}+1,\dots, \mathcal{N}.
\end{cases}
\]
Consequently, the following bounds hold:
\[
h_i \leq
\begin{cases}
	C \mathcal{N}^{-1}, & i = 1, 2, \ldots, \dfrac{\mathcal{N}}{2}, \\[6pt]
	C (\mu \mathcal{N}^{-1} \ln \mathcal{N}), &  i= \dfrac{\mathcal{N}}{2}+1,\dots, \mathcal{N}.
\end{cases}
\]
An analogous construction applies to $I_y$, and similar bounds hold for the mesh sizes $h_j$ in the $y$-direction.
\subsection{Preliminaries and notation for the discontinuous Galerkin method}\label{sec(2.2.3)}
In this subsection, we introduce the notation that will be used throughout the paper.

Let $\mathcal{T}_{\mathcal{N}}^x = \{\mathcal{K}_{x_{i}} = (x_{i-1}, x_{i}) : i = 1, \ldots, \mathcal{N}\}$ and $\mathcal{T}_{\mathcal{N}}^y = \{\mathcal{K}_{y_{j}} = (y_{j-1}, y_{j}) : j = 1, \ldots, \mathcal{N}\}$ denote partitions of the spatial domains $I_x$ and $I_y$, respectively. We denote by $\mathcal{T}_N$ the corresponding tensor-product partition of $I = I_x \times I_y$.

The $L^{2}$-inner product and norm on $I$ are defined elementwise by $\langle \tilde{v}, \tilde{w} \rangle_{I} = \sum_{\mathcal{K} \in \mathcal{T}_N} \langle \tilde{v}, \tilde{w} \rangle_{\mathcal{K}},
\quad
\|\tilde{v}\|_{I}^2 = \sum_{\mathcal{K} \in \mathcal{T}_N} \|\tilde{v}\|_{\mathcal{K}}^2,$ where $\langle \tilde{v}, \tilde{w} \rangle_{\mathcal{K}} = \int_{\mathcal{K}} \tilde{v} \tilde{w} \, dx \, dy.$ Similarly, in the $x$- and $y$-directions, we define $
\langle v, w \rangle_{I_x} = \sum_{i=1}^{\mathcal{N}} \langle \tilde{v}, \tilde{w} \rangle_{\mathcal{K}_{x_{i}}},
\
\|\tilde{v}\|_{I_x}^2 = \sum_{i=1}^{\mathcal{N}} \|\tilde{v}\|_{\mathcal{K}_{x_{i}}}^2, \ \langle \tilde{v}, \tilde{w} \rangle_{I_y} = \sum_{j=1}^{\mathcal{N}} \langle \tilde{v}, \tilde{w} \rangle_{\mathcal{K}_{y_{j}}},
\
\|\tilde{v}\|_{I_y}^2 = \sum_{j=1}^{\mathcal{N}} \|\tilde{v}\|_{\mathcal{K}_{y_{j}}}^2, \sum_{i=1}^{\mathcal{N}}\|\tilde{v}\|_{n,\infty,\mathcal{K}_{x_i}}=\sum_{i=1}^{\mathcal{N}}\|(\tilde{v})^{n}\|_{\infty,\mathcal{K}_{x_i}},$\ and $\sum_{j=1}^{\mathcal{N}}\|\tilde{v}\|_{n,\infty,\mathcal{K}_{y_j}}=\sum_{j=1}^{\mathcal{N}}\|(\tilde{v})^{n}\|_{\infty,\mathcal{K}_{y_j}}.$

Associated with these partitions, we define the broken Sobolev spaces of order $s$ as
\[
H^s(I_x, \mathcal{T}_{\mathcal{N}}^x) = \left\{ \tilde{v} \in L^{2}(I_x) : \tilde{v}|_{\mathcal{K}_{x_{i}}} \in H^s(\mathcal{K}_{x_{i}}), \ \forall \mathcal{K}_{x_{i}} \in \mathcal{T}_{\mathcal{N}}^x \right\},
\]
\[
H^s(I_y, \mathcal{T}_{\mathcal{N}}^y) = \left\{ \tilde{v} \in L^{2}(I_y) : \tilde{v}|_{\mathcal{K}_{y_{j}}} \in H^s(\mathcal{K}_{y_{j}}), \ \forall \mathcal{K}_{y_{j}} \in \mathcal{T}_{\mathcal{N}}^y \right\}.
\]
Based on these meshes, we define the finite element spaces in the $x$- and $y$-directions as
\[
\mathcal{W}_{\mathcal{N}}^k(I_x) = \left\{ \tilde{v} \in L^{2}(I_x) : \tilde{v}|_{\mathcal{K}_{x_{i}}} \in \mathbb{P}^k(\mathcal{\mathcal{K}}_{x_{i}}), \ \forall \mathcal{K}_{x_{i}} \in \mathcal{T}_{\mathcal{N}}^x \right\},
\]
\[
\mathcal{W}_{\mathcal{N}}^k(I_y) = \left\{ \tilde{v} \in L^{2}(I_y) : \tilde{v}|_{\mathcal{K}_{y_{j}}} \in \mathbb{P}^k(\mathcal{K}_{y_{j}}), \ \forall \mathcal{K}_{y_{j}} \in \mathcal{T}_{\mathcal{N}}^y\right\},
\]
where $\mathbb{P}^k(\mathcal{K})$ denotes the space of polynomials of degree at most $k$ on an element $\mathcal{K}$.

We define the jump and average operators separately in the $x$- and $y$-directions.
\vskip2mm
In the $x$-direction, for an interior node $x_{i}$, we define
$[\tilde{w}]_{x_{i}}=[\tilde{w}(x_{i})]=\tilde{w}^{+}(x_{i})-\tilde{w}^{-}(x_{i}), \ \{\tilde{w}\}_{x_{i}}=\{\tilde{w}(x_{i})\}=\frac{1}{2}(\tilde{w}^{+}(x_{i})+\tilde{w}^{-}(x_{i})),\ i=1,\ldots,\mathcal{N}-1,$

and at boundary points $[\tilde{w}(x_{0})]=\tilde{w}(x_{0}),$ $\{\tilde{w}(x_{0})\}=\tilde{w}(x_{0}),$ $[\tilde{w}(x_{\mathcal{N}})]=-\tilde{w}(x_{\mathcal{N}}),\, \{\tilde{w}(x_{\mathcal{N}})\}=\tilde{w}(x_{\mathcal{N}}),$
where $\tilde{w}(x^{+}_{i})=\lim_{x\to x^{+}_{i}}\tilde{w}(x)$ and  $ \tilde{w}(x_{i}^{-})=\lim_{x\to x^{-}_{i}}\tilde{w}(x).$
Similarly, in the $y$-direction, the jump and average operators at an interior node $y_{j}$ are defined in an analogous manner.
\section{Discontinuous Galerkin approximation}\label{sec(2.3)}
In this section, we develop the fully discrete DDG scheme by combining the DDG method for spatial discretization with the backward Euler method for temporal discretization. In addition, we rigorously analyze the stability of the scheme and establish its  coercivity properties.
\subsection{Fully-discrete formulation}\label{sec(2.3.1)}
In this subsection, we first formulate the semi-discrete problems corresponding to \eqref{E3} and \eqref{E4}. At each time step $t=t_{m}$, $1 \leq m \leq M$, we obtain:
\vskip 2mm
\textbf{Step 1 ($x$- direction)}
\begin{equation}\label{E5}
	\left\{
	\begin{array} {ll}
		\text{Find} \quad v^{m} : I\times(t_{m-1},t_{m}]  \to \mathbb{R} \quad \text{so that} \quad  \forall y \in I_y , \\[8pt] 
		\frac{\partial v^{m}}{\partial t} +\mathcal{L}_{x,\mu}v^{m} = r_{1}(x,y,t_{m}),  \quad in \quad I, \\[8pt]
		v^{m}(x,y)=0, \quad  (x,y) \in \partial I_{x}\times I_{y},\\[8pt]
		v^{m-1}(x,y)=u^{m-1}(x,y), \quad (x,y) \in I.
	\end{array}\right.
\end{equation}
\vskip 2mm
\begin{equation}\label{E6}
	\left\{
	\begin{array} {ll}
		\text{Find} \quad v^{m} : I   \to \mathbb{R} \quad \text{so that} \quad  \forall y \in I_y , \\[8pt] 
		\bigg(\frac{v^{m}-v^{m-1}}{\Delta t}\bigg) +\mathcal{L}_{x,\mu}v^{m} = r_{1}(x,y,t_{m}),  \quad in \quad I, \\[8pt]
		v^{m}(x,y)=0, \quad  (x,y) \in \partial I_{x}\times I_{y},\\[8pt]
		v^{m-1}(x,y)=u^{m-1}(x,y), \quad (x,y) \in I.
	\end{array}\right.
\end{equation}
\vskip 2mm
\textbf{Step 2 ($y$- direction)}
\begin{equation}\label{E7}
	\left\{
	\begin{array} {ll}
		\text{Find} \quad w^{m} : I \times(t_{m-1},t_{m}]  \to \mathbb{R} \quad \text{so that} \quad  \forall x \in I_x , \\[8pt] 
		\frac{\partial w^{m}}{\partial t} +\mathcal{L}_{y,\mu}w^{m} = r_{2}(x,y,t_{m}),  \quad in \quad I, \\[8pt]
		w^{m}(x,y)=0, \quad  (x,y) \in  I_{x}\times\partial I_{y},\\[8pt]
		w^{m-1}(x,y)=v^{m}(x,y), \quad (x,y) \in I.
	\end{array}\right.
\end{equation}
\vskip 2mm
\begin{equation}\label{E8}
	\left\{
	\begin{array} {ll}
		\text{Find} \quad w^{m} : I  \to \mathbb{R} \quad \text{so that} \quad  \forall x \in I_x , \\[8pt] 
		\bigg(\frac{w^{m}-w^{m-1}}{\Delta t}\bigg)+\mathcal{L}_{y,\mu}w^{m} = r_{2}(x,y,t_{m}),  \quad in \quad I, \\[8pt]
		w^{m}(x,y)=0, \quad  (x,y) \in  I_{x}\times\partial I_{y},\\[8pt]
		w^{m-1}(x,y)=v^{m}(x,y), \quad (x,y) \in I.
	\end{array}\right.
\end{equation}
Applying the DDG method to \eqref{E6} and \eqref{E8} yields the following fully discrete scheme.
\vskip 2mm
\textbf{Step 1 ($x$- direction)}
\begin{equation}\label{E10}
	\left\{
	\begin{array} {ll}
		\text{Find} \quad v_{h}^{m}(x,y_{j})\in \mathcal{W}_{\mathcal{N}}^k(I_x) \quad  \forall y_{j}  \in \overline{I}_y, \qquad \mbox{such that}\\[8pt]
		\bigg(\frac{v^{m}_h-v^{m-1}_h}{\Delta t},\phi_h\bigg)+B_{x}(v^{m}_{h},\phi_h)=(r^{m}_{1},\phi_h) \quad \forall \phi_h \in \mathcal{W}_{\mathcal{N}}^k(I_x),  \\[8pt] 
	\end{array}\right.
\end{equation}
where 
\begin{eqnarray}
&&B_{x}(v_{h}^{m},\phi_h)=\mu\sum_{i=1}^{\mathcal{N}}\int_{\mathcal{K}_{x_{i}}}\frac{\partial{v_{h}^{m}}}{\partial{x}}\frac{\partial{\phi_h}}{\partial{x}}dx-\sum_{i=1}^{\mathcal{N}}\int_{\mathcal{K}_{x_{i}}}b_{1}(x,y_{j})v_{h}^{m}\frac{\partial{\phi_h}}{\partial{x}}dx\nonumber\\&&+\sum_{i=1}^{\mathcal{N}}\int_{\mathcal{K}_{x_{i}}}\bigg(a_{1}(x,y_{j})-\frac{\partial b_{1}(x,y_{j})}{\partial x}\bigg)v_{h}^{m}\phi_hdx\nonumber\\&&+\mu\sum_{i=0}^{\mathcal{N}}\bigg(\widehat{v_{h}^{m}}|_{i}[\phi_h]_{i}+[v_{h}^{m}]_{i}\bigg\{\frac{\partial{\phi_h}}{\partial{x}}\bigg\}_{i}\bigg)-\sum_{i=0}^{\mathcal{N}}\big(b_{1}(x,y_{j})\widetilde{v_{h}^{m}}[\phi_h]\big)_{i},\nonumber
\end{eqnarray}
\textbf{Step 2 ($y$-direction)}
\begin{equation}\label{E12}
	\left\{
	\begin{array} {ll}
		\text{Find} \quad w_{h}^{m}(x_{i},y)\in \mathcal{W}_{\mathcal{N}}^k(I_y) \quad  \forall x_{i}  \in \overline{I}_x, \qquad \mbox{such that}\\[8pt]
		\bigg(\frac{w^{m}_h-w^{m-1}_h}{\Delta t},\psi_h\bigg)+B_{y}(w^{m}_{h},\psi_h)=(r^{m}_{2},\psi_h) \quad \forall \psi_h \in \mathcal{W}_{\mathcal{N}}^k(I_y),  \\[8pt] 
	\end{array}\right.
\end{equation}
where
\begin{eqnarray}
	&&B_{y}(w^{m}_{h},\psi_h)=\mu\sum_{j=1}^{\mathcal{N}}\int_{\mathcal{K}_{y_{j}}}\frac{\partial{w_{h}^{m}}}{\partial{y}}\frac{\partial{\psi_h}}{\partial{y}}dy-\sum_{j=1}^{\mathcal{N}}\int_{\mathcal{K}_{y_{j}}}b_{2}(x_{i},y)w_{h}^{m}\frac{\partial{\psi_h}}{\partial{y}}dy\nonumber\\&&+\sum_{j=1}^{\mathcal{N}}\int_{\mathcal{K}_{y_{j}}}\bigg(a_{2}(x_i,y)-\frac{\partial b_{2}(x_i,y)}{\partial y}\bigg)w_{h}^{m}\psi_hdy\nonumber\\&&+\mu\sum_{j=0}^{\mathcal{N}}\bigg(\widehat{w_{h}^{m}}|_{j}[\psi_h]_{j}+[w_{h}^{m}]_{j}\bigg\{\frac{\partial{\psi_h}}{\partial{y}}\bigg\}_{j}\bigg)-\sum_{j=0}^{\mathcal{N}}\big(b_{2}(x_{i},y)\widetilde{w_{h}^{m}}[\psi_h]\big)_{j},\nonumber
\end{eqnarray}
where the additional terms $[v_{h}^{m}]_{j}$ and $[w_{h}^{m}]_{j}$ are interface correction terms, which are essential for achieving optimal accuracy of the DDG method when higher-order polynomial approximations are employed. Following \cite{liu2009direct}, the numerical fluxes $\widehat{v^m}$ and $\widehat{w^m}$ are defined, respectively, as
\begin{equation}\label{flux1}
	\widehat{v^{m}}\big|_{i}= \left\{
	\begin{array}{ll}
		\displaystyle

		\eta_{0}\frac{[v^{m}]_{i}}{h_{i}}
		+\bigg\{\frac{\partial v^{m}}{\partial x}\bigg\}_{i}
		+\eta_{1}h_{i}\bigg[\frac{\partial^{2} v^{m}}{\partial x^{2}}\bigg]_{i},
		& i=1,\ldots,\mathcal{N}-1,\\[10pt]
		
		\frac{\partial v^{m}_{0}}{\partial x}, & i=0,\\[6pt]
		
		\frac{\partial v^{m}_{\mathcal{N}}}{\partial x}, & i=\mathcal{N}.
	\end{array}
	\right.
\end{equation}
\begin{equation}\label{flux2}
	\widehat{w^{m}}\big|_{j}= \left\{
	\begin{array}{ll}
		\displaystyle

		\eta_{0}\frac{[w^{m}]_{j}}{h_{j}}
		+\bigg\{\frac{\partial w^{m}}{\partial y}\bigg\}_{j}
		+\eta_{1}h_{j}\bigg[\frac{\partial^{2} w^{m}}{\partial y^{2}}\bigg]_{j},
		& j=1,\ldots,\mathcal{N}-1,\\[10pt]
		
		\frac{\partial w^{m}_{0}}{\partial y}, & j=0,\\[6pt]
		
		\frac{\partial w^{m}_{\mathcal{N}}}{\partial y}, & j=\mathcal{N}.
	\end{array}
	\right.
\end{equation}
where $\eta_0$ and $\eta_1$ are positive user-defined parameters that ensure the stability of the DDG method. We set $\eta_{1}=\frac{1}{2k^{2}+2k}$, as given in {\cite[Lemma $3.2$]{Cao_DDG}}.

The numerical fluxes $\widetilde{u^{m}}\big|_{i}$ and $\widetilde{w^{m}}\big|_{j}$ are defined as
\begin{equation}\label{flux3}
	\widetilde{v^{m}}\big|_{i}=
	\left\{
	\begin{array}{ll}
		\theta\,v^{m}(x_i^-)+(1-\theta)\,v^{m}(x_i^+),
		& i=1,\ldots,\mathcal{N}-1,\\[8pt]
		0,
		& i=0,\\[6pt]
		v^{m}_{\mathcal{N}},
		& i=\mathcal{N},
	\end{array}
	\right.
\end{equation}

\begin{equation}\label{flux4}
	\widetilde{w^{m}}\big|_{j}=
	\left\{
	\begin{array}{ll}
		\theta\,w^{m}(y_j^-)+(1-\theta)\,w^{m}(y_j^+),
		& j=1,\ldots,\mathcal{N}-1,\\[8pt]
		0,
		& j=0,\\[6pt]
		w^{m}_{\mathcal{N}},
		& j=\mathcal{N},
	\end{array}
	\right.
\end{equation}
where we take $ \frac{1}{2}\leq\theta\leq 1$ when
$b_{1}(x,y_j)>0$ and $b_{2}(x_i,y)>0$.
\subsection{Admissibility of the DDG method}\label{sec(2.3.2)}
As is standard in DG methods, the numerical flux is selected to ensure stability. The stability of the proposed DDG method is established through the following definition, which is taken from \cite[Definition $2.1$ ]{liu2009direct}.
\begin{definition}\label{def:admissible}
	The numerical fluxes $\widehat{v^{m}}$ and $\widehat{w^{m}}$ of the forms~\eqref{flux1} and~\eqref{flux2} are said to be admissible if there exist constants $\gamma_1 \in (0,1)$ and $\gamma_2 \in (0,1]$ such that
	\begin{eqnarray*}
		\gamma_1 \sum_{i=1}^{\mathcal{N}}\int_{\mathcal{K}_{x_i}} ((v^{m})')^2\,dx
		+\sum_{i=0}^{\mathcal{N}}\left(\widehat{v^{m}}[v^{m}]+[v^{m}]\{(v^{m})'\}\right)\Big|_{i}
		\ge
		\gamma_2\sum_{i=0}^{\mathcal{N}}\frac{[v^{m}]_i^{\,2}}{h_i}, \forall\, v^{m}\in \mathcal{W}_{\mathcal{N}}^{k}(I_{x}).
	\end{eqnarray*}
	\begin{equation*}
		\gamma_1 \sum_{j=1}^{\mathcal{N}}\int_{\mathcal{K}_{y_j}} ((w^{m})')^2\,dy
		+\sum_{j=0}^{\mathcal{N}}\left(\widehat{w^{m}}[w^{m}]+[w^{m}]\{(w^{m})'\}\right)\Big|_{j}
		\ge
		\gamma_2\sum_{j=0}^{\mathcal{N}}\frac{[w^{m}]_j^{\,2}}{h_j},
		\ \forall\, w^{m}\in \mathcal{W}_{\mathcal{N}}^{k}(I_{y}).
	\end{equation*}
\end{definition}
\begin{lemma}\label{lemma2}
	Suppose that there exist constants $\gamma_1 \in (0,1)$ and $\gamma_2 \in (0,1]$ such that
	\begin{equation*}
		\eta_0
		\ge
		\gamma_2
		+\frac{1}{\gamma_1}
		\max_{v^{m}\in \mathcal{W}_{\mathcal{N}}^{k}(I_{x})}
		\frac{
			\displaystyle
			\sum_{i=0}^{\mathcal{N}}
			h_i
			\left(
			\{(v^{m})'\}_i
			+\frac{\eta_1}{2}h_i[(v^{m})'']_i
			\right)^2
		}{
			\displaystyle
			\sum_{i=1}^{\mathcal{N}}
			\int_{\mathcal{K}_{x_i}}((v^{m})')^2\,\mathrm{d}x
		}.
	\end{equation*}
	An analogous condition holds for the $y$-direction over $\mathcal{W}_{\mathcal{N}}^{k}(I_{y})$. Then the numerical fluxes $\widehat{v^{m}}$ and $\widehat{w^{m}}$ defined by~\eqref{flux1} and~\eqref{flux2} are admissible in the $x$- and $y$-directions, respectively.
\end{lemma}
\begin{proof}
	For the complete proof, see \cite[Lemma~1]{Stynes_DDG}.
\end{proof}
\begin{lemma}[{\cite[Lemma 2.2]{liu2010direct}}]
	Let $k\in\mathbb{N}$, and consider a uniform mesh on the interval $[0,1]$. The numerical flux \eqref{flux1} and \eqref{flux2} with $\eta_1=0$ is admissible provided that
	\begin{equation*}
		\eta_0 \geq \gamma_2+\frac{1}{4\gamma_1}
		\lambda_{\max}\!\left(H^{-1/2}OH^{-1/2}\right),
		\label{eq:beta0_condition}
	\end{equation*}
	where $H$ is the $k\times k$ Hilbert matrix with entries
	\[
	H_{i,j}=\frac{1}{i+j-1}, \qquad i,j=1,\ldots,k,
	\]
	$O$ is the $k\times k$ matrix whose entries are all equal to $1$, and
	$\lambda_{\max}(\cdot)$ represents the largest eigenvalue of
	$H^{-1/2}OH^{-1/2}$.
\end{lemma}\label{lemma3}
As one possible choice, set $\gamma_2=1,\quad \gamma_1=\frac12,$ and
\begin{equation}
	\eta_0=\bigg[
	\frac12\lambda_{\max}\!\left(H^{-1/2}OH^{-1/2}\right)
	\bigg]+2,
	\label{eq:beta0_choice}
\end{equation}
where $[.]$ denotes the integer part. Numerical calculations show
that in \eqref{eq:beta0_choice}, $\eta_0=\bigg[\frac{k^2}{2}\bigg]+2,
\quad k=1,2,\ldots,9,$
 as reported in Table~\ref{table_1}.
\begin{table}[htbp]
	\centering
	\caption{\it{Analytical estimation on $\eta_0(k)$.}}\label{table_1}
	
	\begin{tabular}{@{}lllllllllll@{}}
		\hline
		&	$k$  &   1 &  2 &  3 & 4 & 5 &6 &  7& 8 & 9\\
		\hline
		&	$\eta_0$  & 2 & 4 & 6 & 10 & 14 & 20 & 26 & 34 &  42  \\
		\hline
	\end{tabular}
\end{table}
\subsection{Coercivity and stability analysis of the DDG method}\label{sec(2.3.3)} 
In this subsection, we establish the coercivity, and stability of the proposed numerical schemes. To facilitate the analysis, we introduce the energy norm and the DG norm. The energy norm is induced by the bilinear form of the underlying differential operator, whereas the DG norm incorporates both the solution and the jumps across element interfaces. Since the DG norm is slightly stronger than the energy norm, it is employed throughout this work for the error and convergence analysis.
The energy norm and the DG norm in the $x$-direction are defined as
\begin{equation}\label{DG_norm}
	\left\{
	\begin{array}{ll}
		\|v^{m}\|_{E}^{2}
		&=
		|v^{m}|_{1}^{2}
		+
		\displaystyle\sum_{i=0}^{\mathcal{N}}
		\frac{\eta_{0}}{h_i}[v^{m}]_{i}^{2},
		\qquad
		\|v^{m}\|_{\mathrm{DG}}^{2}
		=
		\mu \|v^{m}\|_{E}^{2}
		+
		\|v^{m}\|_{0}^{2}.
		\\[2mm]
		
	\end{array}
	\right.
\end{equation}
where
\begin{align*}
	\|v^{m}\|_{0}^{2}
	&=
	\sum_{i=1}^{\mathcal{N}}\int_{\mathcal{K}_{x_i}}(v^{m})^{2}\,dx,
	&
	|v^{m}|_{1}^{2}
	&=
	\sum_{i=1}^{\mathcal{N}}\int_{\mathcal{K}_{x_i}}(v_{x}^{m})^{2}\,dx.
\end{align*}
The DG norm $\|\cdot\|_{\mathrm{DG}}$ in the $y$-direction is defined analogously.
\begin{lemma}[Coercivity]\label{Coercivity}
	Let the numerical fluxes \eqref{flux1} and \eqref{flux2} satisfy the admissibility conditions. Then there exists a constant $\lambda>0$, independent of the perturbation parameter and the mesh size, such that
	\begin{align}
		B_x(\phi_h^{m},\phi_h^{m}) &\ge \lambda \|\phi_h^{m}\|_{DG}^{2},
		&&\forall\,\phi_h^{m}\in \mathcal{W}_{\mathcal N}^{k}(I_x),\nonumber\\
		B_y(\psi_h^{m},\psi_h^{m}) &\ge \lambda \|\psi_h^{m}\|_{DG}^{2},
		&&\forall\,\psi_h^{m}\in \mathcal{W}_{\mathcal N}^{k}(I_y).\nonumber
	\end{align}
\end{lemma}
\begin{proof}
	The details of the proof can be found in \cite[Lemma $3$]{Stynes_DDG}.
	\end{proof}
\begin{lemma}[Stability]\label{Stability}
	Let $v^{m}_{h}$ and $w^{m}_{h}$ represent the solution of Equations \eqref{E10} and \eqref{E12}, respectively. Then at the time level $t=t_{m}$, for $1\leq m\leq M$. The solution at each time step satisfies the following stability bound:
	\begin{eqnarray}
		&&\sum_{j=0}^{\mathcal{N}}\big\|w^{m}_{h}(x_{j},y)\big\|_{L^{2}(I_y)}^{2}+	\Delta t\sum_{i=0}^{m}\sum_{j=0}^{\mathcal{N}}\bigg\{\big\|v^{i}_{h}(x,y_{j})\big\|_{DG}^{2}+\big\|w^{i}_{h}(x_{j},y)\big\|_{DG}^{2}\bigg\}\nonumber\\
		&&\leq\bigg(\sum_{j=0}^{\mathcal{N}}\big\|v^{0}_{h}(x,y_{j})\big\|_{L^{2}(I_x)}^{2}+\Delta t \sum_{i=0}^{m}\sum_{j=0}^{\mathcal{N}} \bigg\{\big\|r^{i}_{1}(x,y_{j})\big\|_{L^{2}(I_x)}^{2}+\big\|r^{i}_{2}(x_{j},y)\big\|_{L^{2}(I_y)}^{2}\bigg\} \bigg).\nonumber
	\end{eqnarray}
\end{lemma}
\begin{proof}
	Taking  $\phi_h=v^{m}_{h}(x,y_{j})$,  $\forall y_{j}\in \overline{I}_{y}$ in Equation \eqref{E10}, we obtain
	\begin{eqnarray}
		\bigg(\frac{v_{h}^{m}-v_{h}^{m-1}}{\Delta t},v^{m}_{h}(x,y_{j})\bigg)_{L^{2}(I_x)}+B_{x}\big(v^{m}_{h},v^{m}_{h}(x,y_{j})\big)=\big(r^{m}_{1}(x,y_{j}),v^{m}_{h}(x,y_{j})\big)_{L^{2}(I_x)}.\nonumber
	\end{eqnarray}
	From Lemma \ref{Coercivity} and the  Cauchy-Schwarz inequality, we have
	\begin{eqnarray}
		&&\frac{1}{2\Delta t}\bigg(\|v_{h}^{m}\|_{L^{2}(I_x)}^{2}-\|v_{h}^{m-1}\|^{2}_{L^{2}(I_x)}+\|v_{h}^{m}-v_{h}^{m-1}\|_{L^{2}(I_x)}^{2}\bigg)+\|v_{h}^{m}\|_{DG}^{2}\nonumber\\
		&&\leq\frac{1}{2}\bigg(\|r^{m}_{1}(x,y_{j})\|_{L^{2}(I_x)}^{2}+ \|v^{m}_{h}(x,y_{j})\|_{L^{2}(I_x)}^{2}\bigg).\nonumber
	\end{eqnarray}
	Using the fact that $\|v_{h}^{m}-v_{h}^{m-1}\|_{L^{2}(I_x)}^{2}\geq 0$ and $\|v^{m}_{h}(x,y_{j})\|_{L^{2}(I_x)}^{2} \leq \|v^{m}_{h}(x,y_{j})\|_{DG}^{2}$,  we have
	\begin{eqnarray}
		&&\bigg(\|v_{h}^{m}(x,y_{j})\|_{L^{2}(I_x)}^{2}-\|v_{h}^{m-1}(x,y_{j})\|_{L^{2}(I_x)}^{2}\bigg)+2\Delta t\|v_{h}^{m}(x,y_{j})\|_{DG}^{2}\nonumber\\
		&&\leq\Delta t\bigg(\|r^{m}_{1}(x,y_{j})\|_{L^{2}(I_x)}^{2}+ \|v^{m}_{h}(x,y_{j})\|_{DG}^{2}\bigg).\nonumber
	\end{eqnarray}
	In the $x$-direction, the above inequality can be expressed as:
	\begin{eqnarray}\label{E19}
		\bigg(\|v_{h}^{m}\|_{L^{2}(I_x)}^{2}\bigg)+\Delta t\|v_{h}^{m}\|_{DG}^{2}\leq\bigg(\|v_{h}^{m-1}\|_{L^{2}(I_x)}^{2}+\Delta t\|r^{m}_{1}(x,y_{j})\|_{L^{2}(I_x)}^{2}\bigg).
	\end{eqnarray}
	Similarly, the following estimate holds in the $y$- direction:
	\begin{eqnarray}\label{E20}
		\bigg(\|w^{m}_h\|_{L^{2}(I_y)}^{2}\bigg)+\Delta t\|w^{m}_h\|_{DG}^{2}\leq\bigg(\|w^{m-1}_h\|_{L^{2}(I_y)}^{2}+\Delta t\|r^{m}_{2}(x_{j},y)\|_{L^{2}(I_y)}^{2}\bigg).
	\end{eqnarray}
	Adding Equations \eqref{E19} and \eqref{E20} and summing over $i$ and $j$, we obtain the desired stability result.
\end{proof}
\section{Error analysis}\label{sec(2.4)}
In the FEM, the total error is commonly decomposed into two components, namely the interpolation error $(u^{m}- Pu^{m})$ and the discretization error $(Pu^{m}-u^{m}_{h})$. In this section, we first derive an estimate for the interpolation error in the DG norm. Subsequently, this result is utilized to establish a bound for the discretization error corresponding to the proposed finite element formulation. Finally, by combining these estimates, we obtain the overall error bound for the numerical scheme.
\[
\|u^{m}-u^{m}_{h}\|_{DG}\leq \|u^{m}-Pu^{m}\|_{DG}+\|Pu^{m}-u^{m}_{h}\|_{DG}.
\]
Next, we focus on deriving error estimates for Equation \eqref{E10} and \eqref{E12}. The total error is further decomposed into contributions arising from the spatial directions as follows:
\begin{eqnarray}
	&&\|u^{m}-u^{m}_{h}\|_{DG} \leq \|v^{m}(x,y_{j})-v^{m}_{h}(x,y_{j})\|_{DG} + \|w^{m}(x_{i},y)-w^{m}_{h}(x_{i},y)\|_{DG},\nonumber\\
	&&\|\mathcal{E}^{m}_{h,i,j}\|_{DG}\leq \|E^{* m}_{h,j}\|_{DG}+\|E^{\dagger m}_{h,i}\|_{DG}, \quad i,j=0,1,\ldots,\mathcal{N},\nonumber
\end{eqnarray}
where $E^{*m}_{h,j}$ and $E^{\dagger m}_{h,i}$  denote the errors of the $x$- and $y$-directions, respectively, at time level $t=t_{m}$.

Each directional error is further decomposed into interpolation and discretization parts. For the $x$-direction, we have
\begin{eqnarray}
	\|E^{* m}_{h,j}\|_{DG}&& \leq \|v^{m}(x,y_{j})-Pv^{m}(x,y_{j})\|_{DG} + \|Pv^{m}(x,y_{j})-v^{m}_{h}(x,y_{j})\|_{DG}\nonumber\\
	&& \leq \|\rho^{* m}\|_{DG} + \|\xi^{* m}\|_{DG},\quad j=0,1,\ldots,\mathcal{N},\nonumber
\end{eqnarray}
where $\rho^{*m}$ and $\xi^{*m}$ represent the interpolation and discretization errors, respectively.

Similarly, for the $y$-direction, we obtain
\begin{eqnarray}
	\|E^{\dagger m}_{h,j}\|_{DG} &&\leq \|w^{m}(x_{i},y)-Pw^{m}(x_{i},y)\|_{DG} + \|Pw^{m}(x_{i},y)-w^{m}_{h}(x_{i},y)\|_{DG}\nonumber\\
	&&\leq \|\rho^{\dagger m}\|_{DG} + \|\xi^{\dagger m}\|_{DG},\quad i=0,1,\ldots,\mathcal{N},\nonumber
\end{eqnarray}
where $\rho^{\dagger m}$ and $\xi^{\dagger m}$ denote the corresponding interpolation and discretization errors, respectively.
\subsection{Interpolation error}\label{Interpolation}
Let $P$ denote the $L^{2}$-projection from $L^{2}(I_{x})$ and $L^
{2}(I_{y})$ onto the finite element spaces $L^{2}(I_{x})$ and $L^{2}(I_{y})$, respectively. That is, for each $v^{m} \in L^{2}(I_{x})$ and $w^{m} \in L^{2}(I_{y})$, let $Pv^{m} \in \mathcal{W}_{\mathcal{N}}^k(I_x)$ and $Pw^{m} \in \mathcal{W}_{\mathcal{N}}^k(I_y)$ be the unique elements such that
\[
\begin{aligned}
	\int_{\mathcal{K}_{x_i}} \bigl(P v^{m}(x)-v^{m}(x)\bigr)v_h(x)\,dx &= 0,
	&& \forall\, v_h \in \mathcal{W}_{\mathcal{N}}^k(I_x),
	\quad i=1,\ldots,\mathcal{N},\\
	\int_{\mathcal{K}_{y_j}} \bigl(P w^{m}(y)-w^{m}(y)\bigr)w_h(y)\,dy &= 0,
	&& \forall\, w_h \in \mathcal{W}_{\mathcal{N}}^k(I_y),
	\quad j=1,\ldots,\mathcal{N}.
\end{aligned}
\]
\begin{lemma}\label{lemma8}
	Suppose that \(v^{m} \in H_0^1(I_x) \cap H^{k+1}(I_x)\). Then the following holds for all \(l \in \{0,1\}\):
	\begin{eqnarray*}
		&&	\sum_{i=1}^{\frac{\mathcal{N}}{2}}\|Pv^{m}-v^{m}\|^{2}_{l,\mathcal{K}_{x_i}}\leq  C \left[\mu^{1-2l} \mathcal N^{-2(k+1)}+\mathcal N^{-2k-2+2l}\right],\\[1.5ex]
		&&	\sum_{i=\frac{\mathcal{N}}{2}+1}^{\mathcal N}\|Pv^{m}-v^{m}\|^{2}_{l,\mathcal{K}_{x_i}}\leq C \left[\mu^{1-2l} (\mathcal N^{-1}\ln \mathcal{N})^{2k+2-2l}\right].
	\end{eqnarray*}
\end{lemma}
\begin{proof} 
	The proof follows analogously to \cite[Lemma~6]{Stynes_DDG}.
\end{proof}
\begin{lemma}\label{lemma9}
	Suppose that \(v^{m} \in H_0^1(I_x) \cap H^{k+1}(I_x)\). Then the following holds for all $n \in \{0, 1, 2\}$:
	\begin{align*}
		\sum_{i=1}^{\frac{\mathcal{N}}{2}} (h_i)^{2n-1} \|Pv^{m} - v^{m}\|^2_{n,\infty,\mathcal{K}_{x_i}} 
		&\leq C \left[ \mathcal N^{-2k}+\mu \mathcal N^{-2k}+ \mu^{1-2n}\mathcal N^{-2k-2n}, \right] \\[1.5ex]
		\sum_{i=\frac{\mathcal{N}}{2}+1}^{\mathcal{N}} (h_i)^{2n-1} \|Pv^{m} - v^{m}\|^2_{n,\infty,\mathcal{K}_{x_i}} 
		&\leq C \left[ (\mu \mathcal{N}^{-1} \ln \mathcal{N})^{2k}+ \mu^{-1} (\mathcal{N}^{-1} \ln \mathcal{N})^{2k} \right].
	\end{align*}
\end{lemma}
\begin{proof}
	The proof follows using techniques similar to those presented in \cite[Lemma~$7$]{Stynes_DDG}. 
\end{proof}
\begin{theorem}\label{theorem1}
	Suppose that \(v^{m} \in H_0^1(I_x) \cap H^{k+1}(I_x)\), \(w^{ m} \in H_0^1(I_y) \cap H^{k+1}(I_y)\) denotes the exact solution of Equation~(\ref{E3}) and (\ref{E4}), respectively, at the $m$-th time step. Then the interpolation errors satisfy
	\begin{eqnarray}
		\|Pv^{m} - v^{m}\|_{DG} &\leq& C\big(\mathcal{N}^{-1}\ln \mathcal{N}\big)^{k}, \nonumber\\
		\|Pw^{m} - w^{m}\|_{DG}  &\leq& C\big(\mathcal{N}^{-1}\ln \mathcal{N}\big)^{k}. \nonumber
	\end{eqnarray}
\end{theorem}
\begin{proof}Based on the DG norm defined in Equation~\eqref{DG_norm}, the error can be written as follows:
	\begin{eqnarray}
		&&\|Pv^{m}-v^{m}\|_{\mathrm{DG}}^2
		=
		\mu \|Pv^{m}-v^{m}\|_{E}^2
		+
		\|Pv^{m}-v^{m}\|_{0}^2 \nonumber\\
		&&=
		\mu |Pv^{m}-v^{m}|_{1}^2
		+
		\mu\sum_{i=0}^{\mathcal{N}}\frac{\eta_{0}}{h_i}[Pv^{m}-v^{m}]_{i}^2
		+
		\|Pv^{m}-v^{m}\|_{0}^2.
		\label{eq:DG_error}
	\end{eqnarray}
	The bounds for the first and third terms in Equation~\eqref{eq:DG_error} can be established using Lemma~\ref{lemma8}. The estimate for the second term in Equation~\eqref{eq:DG_error} is derived as follows:
	\begin{eqnarray*}
		\sum_{i=0}^{\mathcal{N}}\frac{\eta_{0}}{h_i}[Pv^{m}-v^{m}]_{i}^2\leq&&	\sum_{i=0}^{\frac{\mathcal{N}}{2}}(h_i)^{-1}\|Pv^{m}-v^{m}]\|^{2}_{0,\infty,\mathcal{K}_{x_i}}\nonumber\\
		&&+\sum_{\frac{\mathcal{N}}{2}+1}^{\mathcal{N}}(h_i)^{-1}\|Pv^{m}-v^{m}\|^{2}_{0,\infty,\mathcal{K}_{x_i}}
	\end{eqnarray*}
	Using Lemma~\ref{lemma9} with $n=0$, we obtain
	\[	\sum_{i=0}^{\mathcal{N}}\frac{\eta_{0}}{h_i}[Pv^{m}-v^{m}]_{i}^2\leq C\mu^{-1} (\mathcal{N}^{-1} \ln \mathcal{N})^{2k}.\]
	By substituting the bounds for all terms into Equation~(\ref{eq:DG_error}), we obtain
	\begin{eqnarray}
		\|Pv^{m} - v^{m}\|_{DG} &\leq& C\big(\mathcal{N}^{-1}\ln \mathcal{N}\big)^{k}. \nonumber
	\end{eqnarray}
	Similarly, we get
	\begin{eqnarray}
		\|Pw^{m} - w^{m}\|_{DG} &\leq& C\big(\mathcal{N}^{-1}\ln \mathcal{N}\big)^{k}. \nonumber
	\end{eqnarray}
	This concludes the proof. 
\end{proof}
\subsection{Discretization error}\label{Discretization}
In this subsection, we derive the bound for the discretization error arising from the discretization of our continuous problem into a finite-dimensional space using the DDG method.
\begin{theorem}\label{theorem2}
	For $1\leq m\leq M$, the following bounds hold at the
	$m$-th time step:
	\[|B_{x}(\rho^{* m},\xi^{* m})|\leq C(\mathcal{N}^{-1}\ln \mathcal{N})^{k}\|\xi^{* m}\|_{DG},\]
	\[|B_{y}(\rho^{\dagger m},\xi^{\dagger m})|\leq C(\mathcal{N}^{-1}\ln \mathcal{N})^{k}\|\xi^{\dagger m}\|_{DG}.\]
\end{theorem}
\begin{proof}
	We have the following 	Let $\widehat{\rho^{*m}}|_{i}=\frac{\partial \rho^{*m}_{i}}{\partial x},$ for $i={0,\mathcal{N}}$ and 
	\begin{eqnarray*}
		&&\widehat{\rho^{*m}}|_{i}=	\eta_{0}\frac{[\rho^{*m}]_{i}}{h_{i}}
		+\bigg\{\frac{\partial \rho^{*m}}{\partial x}\bigg\}_{i}
		+\eta_{1}h_{i}\bigg[\frac{\partial^{2} \rho^{*m}}{\partial x^{2}}\bigg]_{i},\quad
		i=1,\ldots,\mathcal{N}-1,
	\end{eqnarray*}
	\begin{eqnarray*}
		&&B_{x}(\rho^{*m},\xi^{*m})=	\mu\sum_{i=1}^{\mathcal{N}}\int_{\mathcal{K}_{x_{i}}}\frac{\partial{\rho^{*m}}}{\partial{x}}\frac{\partial{\xi^{*m}}}{\partial{x}}dx+\mu\sum_{i=0}^{\mathcal{N}}\bigg(\widehat{\rho^{*m}}|_{i}[\xi^{*m}]_{i}\nonumber\\&&+[\rho^{*m}]_{i}\bigg\{\frac{\partial{\xi^{*m}}}{\partial{x}}\bigg\}_{i}\bigg)-\sum_{i=1}^{\mathcal{N}}\int_{\mathcal{K}_{x_{i}}}b_{1}(x,y_{j})\rho^{*m}\frac{\partial{\xi^{*m}}}{\partial{x}}dx-\sum_{i=0}^{\mathcal{N}}\big(b_{1}(x,y_{j})\widetilde{\rho^{*m}}[\xi^{*m}]\big)_{i}\nonumber\\&&+\sum_{i=1}^{\mathcal{N}}\int_{\mathcal{K}_{x_{i}}}\bigg(a_{1}(x,y_{j})-\frac{\partial b_{1}(x,y_{j})}{\partial x}\bigg)\rho^{*m}\xi^{*m}dx,\nonumber
	\end{eqnarray*}
	For the first term of $B_{x}(\rho^{*m},\xi^{*m})$, the definition of the DG norm together with the interpolation error  estimate yields
	\begin{eqnarray*}
		&&\mu\sum_{i=1}^{\mathcal{N}}\int_{\mathcal{K}_{x_{i}}}\frac{\partial{\rho^{*m}}}{\partial{x}}\frac{\partial{\xi^{*m}}}{\partial{x}}dx\leq C\mu^{1/2}(\mathcal{N}^{-1}\ln \mathcal{N})^{k}\|\xi^{*m}\|_{E}.
	\end{eqnarray*}
	For the Second term of $B_{x}(\rho^{*m},\xi^{*m})$, the definition of the DG norm together with  Lemma~\ref{lemma9}  and the interpolation error  estimate yields
	\begin{align*}
		&\mu\sum_{i=0}^{\mathcal{N}}\left(\widehat{\rho^{*m}}\Big|_{i}[\xi^{*m}]_{i} + [\rho^{*m}]_{i}\left\{\frac{\partial{\xi^{*m}}}{\partial{x}}\right\}_{i}\right) \\
		&\quad \le \mu \left[\sum_{i=1}^{\mathcal{N}}\left(\frac{\eta_{0}}{h_{i}}[\rho^{*m}]^{2}_{i} + h_{i}\left\{\frac{\partial \rho^{*m}}{\partial x}\right\}^{2}_{i} + \frac{\eta^{2}_{1}}{\eta_{0}}h^{3}_{i}\left[\frac{\partial^{2} \rho^{*m}}{\partial x^{2}}\right]^{2}_{i}\right)\right]^{1/2} \\
		&\quad\qquad \times \left[ \left(\sum_{i=1}^{\mathcal{N}}\frac{\eta_{0}}{h_{i}}[\xi^{*m}]^{2}_{i}\right)^{1/2} + \left(\sum_{i=1}^{\mathcal{N}}\frac{h_{i}}{\eta_{0}}\left\{\frac{\partial{\xi^{*m}}}{\partial{x}}\right\}^{2}_{i}\right)^{1/2} \right] \\
		&\quad \le C\mu^{1/2}(\mathcal{N}^{-1}\ln \mathcal{N})^{k}\|\xi^{*m}\|_{E}.
	\end{align*}
	For the third term of $B_{x}(\rho^{*m},\xi^{*m})$, the definition of the DG norm together with the inverse inequality and  the interpolation error  estimate yields
	\begin{align*}
		&\sum_{i=1}^{\mathcal{N}}\int_{\mathcal{K}_{x_{i}}}b_{1}(x,y_{j})\rho^{*m}\frac{\partial{\xi^{*m}}}{\partial{x}}\,dx 
		\le \left[\left(\sum_{i=1}^{\frac{\mathcal{N}}{2}}h^{-2}_i\|\rho^{*m}\|^{2}_{0,\mathcal{K}_{x_{i}}}\right)^{1/2}\left(\sum_{i=1}^{\frac{\mathcal{N}}{2}}h^{2}_i\|\xi^{*m}\|^{2}_{1,\mathcal{K}_{x_{i}}}\right)^{1/2} \right. \nonumber \\
		&\quad \left. + \left(\sum_{i=\frac{\mathcal{N}}{2}+1}^{\mathcal{N}}\|\rho^{*m}\|^{2}_{0,\mathcal{K}_{x_{i}}}\right)^{1/2}\left(\sum_{i=\frac{\mathcal{N}}{2}+1}^{\mathcal{N}}\|\xi^{*m}\|^{2}_{1,\mathcal{K}_{x_{i}}}\right)^{1/2}\right] \nonumber \\
		&\le C\left[\mathcal{N}^{-k}\|\xi^{*m}\|_{0} + \mu^{1/2}(\mathcal{N}^{-1}\ln \mathcal{N})^{k+1}\|\xi^{*m}\|_{E}\right].
	\end{align*}
	For the fourth term of $B_{x}(\rho^{*m},\xi^{*m})$, the definition of the DG norm together with the inverse inequality and the interpolation error  estimate yields
	\begin{eqnarray*}
		\sum_{i=0}^{\mathcal{N}}\big(b_{1}(x,y_{j})\widetilde{\rho^{*m}}[\xi^{*m}]\big)_{i}\leq&& \bigg(\sum_{i=1}^{\frac{\mathcal{N}}{2}}h^{-1}_i(\widetilde{\rho^{*m}})^{2}_i\bigg)^{1/2}\bigg(\sum_{i=1}^{\frac{\mathcal{N}}{2}}h_i[\xi^{*m}]_{i}^{2}\bigg)^{1/2}\nonumber\\
		&& \qquad+\bigg(\sum_{i=\frac{\mathcal{N}}{2}+1}^{\mathcal{N}}h_i(\widetilde{\rho^{*m}})^{2}_i\bigg)^{1/2}\bigg(\sum_{i=\frac{\mathcal{N}}{2}+1}^{\mathcal{N}}\frac{\eta_{0}}{h_{i}}[\xi^{*m}]_{i}^{2}\bigg)^{1/2}\nonumber\\
	\end{eqnarray*}
	\begin{eqnarray*}
		&&\leq \bigg(\sum_{i=1}^{\mathcal{N}}h^{-1}_i\|\rho^{*m}\|^{2}_{0,\infty,\mathcal{K}_{x_{i}}}\bigg)^{1/2}\|\xi^{*m}\|_{0}+\bigg(\sum_{i=\frac{\mathcal{N}}{2}+1}^{\mathcal{N}}h_i\|\rho^{*m}\|^{2}_{0,\infty,\mathcal{K}_{x_{i}}}\bigg)^{1/2}\|\xi^{*m}\|_{E}\nonumber\\
		&&\leq C\left[\mathcal{N}^{-k}\|\xi^{*m}\|_{0}+\mu^{1/2}(\mathcal{N}^{-1}\ln \mathcal{N})^{(k+1)}\|\xi^{*m}\|_{E}\right].
	\end{eqnarray*}
	We now proceed to compute the final term of $B_{x}(\rho^{*m},\xi^{*m})$, the definition of the DG norm together with the interpolation error estimate yields
	\begin{eqnarray*}
		&&\sum_{i=1}^{\mathcal{N}}\int_{\mathcal{K}_{x_{i}}}\bigg(a_{1}(x,y_{j})-\frac{\partial b_{1}(x,y_{j})}{\partial x}\bigg)\rho^{*m}\xi^{*m}dx\leq C\mu^{1/2}(\mathcal{N}^{-1}\ln \mathcal{N})^{(k+1)}\|\xi^{*m}\|_{0}.
	\end{eqnarray*}
	Combining all, we get the required result. 
	\end{proof}
\begin{lemma}\label{L9}
	Let $Pv^{m}$ represent the $L^{2}$-projection corresponding to the solution $v^{m}$ of \eqref{E5} and $v^{m}_{h}(x,y_{j})$
	be the numerical solution of Equation \eqref{E10}, respectively. then for $\phi_h\in \mathcal{W}^{k}_N(I_{x}),$ we have
	\begin{eqnarray*}\label{E25}
		(\partial_{\theta}\xi^{* m},\phi_h)+B_{x}(\xi^{* m},\phi_h)= \bigg(\partial_{\theta}Pv^{m}-\frac{\partial v^{m}}{\partial t},\phi_h\bigg)-B_{x}(\rho^{* m},\phi_h),
	\end{eqnarray*}
	where $\partial_{\theta}\xi^{* m}=\bigg(\frac{\xi^{* m}-\xi^{* m-1}}{\Delta t}\bigg)$.
\end{lemma}
\begin{proof} Testing the first equation of \eqref{E5} by a test function $\phi_h\in \mathcal{W}^{k}_N(I_{x})$, we obtain 
	\[\bigg(\frac{\partial v^{m}}{\partial t},\phi_h\bigg)+B_{x}(v^{m},\phi_h)=(r^{m}_{1},\phi_h)\]
	Adding both sides of the above equation $\bigg(\frac{Pv^{m}-Pv^{m-1}}{\Delta t},\phi_h\bigg)$, we get
	\[\bigg(\frac{Pv^{m}-Pv^{m-1}}{\Delta t},\phi_h\bigg)+B_{x}(v^{m},\phi_h)=(r^{m}_{1},\phi_h)+\bigg(\frac{Pv^{m}-Pv^{m-1}}{\Delta t}-\frac{\partial v^{m}}{\partial t},\phi_h\bigg)\]
	Now subtracting Equation \eqref{E10} from the above equation, we get 
	\[	(\partial_{\theta}\xi^{* m},\phi_h)+B_{x}(\xi^{* m},\phi_h)= \bigg(\partial_{\theta}Pv^{m}-\frac{\partial v^{m}}{\partial t},\phi_h\bigg)-B_{x}(\rho^{* m},\phi_h).\]
	Thus, we conclude the proof.
\end{proof}
\begin{lemma}\label{L10}
	Let $Pw^{m}$ represent the $L^{2}$-projection corresponding to the solution $w^{m}$ of \eqref{E7} and $w^{m}_{h}(x_{i},y)$
	be the numerical solution of Equation \eqref{E12}, respectively. then for $\psi_h\in \mathcal{W}^{k}_N(I_{y}),$ we have
	\begin{eqnarray}
		(\partial_{\theta}\xi^{\dagger m},\psi_h)+B_{y}(\xi^{\dagger m},\psi_h)= \bigg(\partial_{\theta}Pw^{ m}-\frac{\partial w^{ m}}{\partial t},\psi_h\bigg)-B_{y}(\rho^{\dagger m},\psi_h),\nonumber
	\end{eqnarray}
	where $\partial_{\theta}\xi^{\dagger m}=\bigg(\frac{\xi^{\dagger m}-\xi^{\dagger m-1}}{\Delta t}\bigg)$.
\end{lemma}
\begin{proof} By applying the methodology outlined in Lemma \ref{L9}. Thus, we obtain the desired result.
\end{proof}
\begin{theorem}\label{theorem3}
	Let $Pv$ and $Pw$ be the $L^{2}$-projection of the solution $v$ and $w$, respectively, and $v^{m}_{h}(x,y_{j})$ and $w^{m}_{h}(x_{i},y)$ be the numerical solution of the Equation \eqref{E10} and \eqref{E12}, respectively. Assume that $v\in H^{k+1}(I_x)$ and $w\in H^{k+1}(I_y).$  Then we have
	\begin{eqnarray}
		\big\|\xi^{*m}(x,y_{j})\big\|_{0}^{2}+	C\Delta t\sum_{l=1}^{m}\big\|\xi^{* l}(x,y_{j})\big\|_{DG}^{2}\leq C\bigg((\Delta t)^{2}\int_{0}^{t_{m}}\big\|v_{tt}\big\|^{2}ds +(\mathcal{N}^{-1}\ln \mathcal{N})^{2k} \bigg),\nonumber\\
		\big\|\xi^{\dagger m}(x_{i},y)\big\|_{0}^{2}+C	\Delta t\sum_{l=1}^{m}\big\|\xi^{\dagger l}(x_{i},y)\big\|_{DG}^{2}\leq C\bigg((\Delta t)^{2}\int_{0}^{t_{m}}\big\|w_{tt}\big\|^{2}ds +(\mathcal{N}^{-1}\ln \mathcal{N})^{2k} \bigg).\nonumber
	\end{eqnarray}
\end{theorem}
\begin{proof} Taking $\phi_h=\xi^{*m}$, in Lemma~\ref{L9}, we obtain
	\[(\partial_{\theta}\xi^{*m},\xi^{*m})+B_{x}(\xi^{*m},\xi^{*m})= \bigg(\partial_{\theta}Pv^{m}-\frac{\partial v^{m}}{\partial t},\xi^{*m}\bigg)-B_{x}(\rho^{* m},\xi^{*m}).\]
	By the coercivity of $B_{x}(\cdot,\cdot)$ and the Cauchy-Schwarz inequality, we obtain
	\[\|\xi^{*m}\|_{0}^{2}-\|\xi^{*m-1}\|_{0}\|\xi^{*m}\|_{0}+\Delta t\|\xi^{*m}\|_{DG}^{2}\leq \Delta t\bigg(\partial_{\theta}Pv^{m}-\frac{\partial v^{m}}{\partial t},\xi^{*m}\bigg)-\Delta tB_{x}(\rho^{* m},\xi^{*m}). \]
	\[\|\xi^{*m}\|_{0}^{2}-\|\xi^{*m-1}\|_{0}^{2}+2\Delta t\|\xi^{*m}\|_{DG}^{2}\leq 2\Delta t\bigg(\partial_{\theta}Pv^{m}-\frac{\partial v^{m}}{\partial t},\xi^{*m}\bigg)-2\Delta tB_{x}(\rho^{* m},\xi^{*m}). \]
	We simplify the above equation  and, applying Theorem~\ref{theorem2}, we get
	\begin{eqnarray}\label{E26}
		\|\xi^{*m}\|_{0}^{2}-\|\xi^{*m-1}\|_{0}^{2}+2\Delta t\|\xi^{*m}\|_{DG}^{2}
		\leq &&\Delta t\bigg\|\partial_{\theta}Pv^{m}-\frac{\partial v^{m}}{\partial t}\bigg\|_{0}^{2}+\Delta t\|\xi^{*m}\|_{DG}^{2}\nonumber\\
		&&+C2\Delta t(\mathcal{N}^{-1}\ln \mathcal{N})^{k}\|\xi^{*m}\|_{DG}. 
	\end{eqnarray}
	Let $\Phi^{m}=\bigg(\partial_{\theta}Pv^{m}-\frac{\partial v^{m}}{\partial t}\bigg)=\bigg(\partial_{\theta}Pv^{m}-\partial_{\theta}v^{m}\bigg)+\bigg(\partial_{\theta}v^{m}-\frac{\partial v^{m}}{\partial t}\bigg)=T_{1}+T_{2},$
	
	where 
	\begin{eqnarray}\label{E27}
		T^{2}_{1}\leq&&\int_{I_x}\bigg[\frac{1}{\Delta t}\int_{t_{m-1}}^{t_{m}}\frac{\partial}{\partial t}\big(Pv^{}(.,s)-v^{}(.,s)\big)ds\bigg]^{2}dx\nonumber\\
		\leq&&\int_{I_x}\bigg[\frac{1}{\Delta t^{2}}\bigg(\int_{t_{m-1}}^{t_{m}}\bigg|\frac{\partial Pv^{}(.,s)}{\partial t}-\frac{\partial v(.,s)}{\partial t}\bigg|^{2}ds\bigg)\bigg(\int_{t_{m-1}}^{t_{m}}ds\bigg)\bigg]dx\nonumber\\
		\leq&& {C}\bigg\|\frac{\partial Pv^{}(.,s)}{\partial t}-\frac{\partial v^{}(.,s)}{\partial t}\bigg\|^{2}_{L^{\infty}(0,T;L^{2}(I_x))}
		\leq {C}(\mathcal{N}^{-1}\ln \mathcal{N})^{2k+2},
	\end{eqnarray}
	and 
	\begin{eqnarray}\label{E28}
		T^{2}_{2}&&\leq\int_{I_x}\bigg[\frac{1}{\Delta t}\big(v^{m}-v^{m-1}\big)-\frac{\partial v^{m}}{\partial t}\bigg]^{2}dx\nonumber\\
		&&\leq\int_{I_x}\bigg[\frac{-1}{\Delta t}\int_{t_{m-1}}^{t_{m}}\bigg((s-t_{m-1})\frac{\partial^{2}  v^{}(s)}{\partial t^{2}}\bigg)ds\bigg]^{2}dx\nonumber\\
		&&\leq \frac{1}{(\Delta t)^{2}}\int_{I_x}\bigg[\int_{t_{m-1}}^{t_{m}}(s-t_{m-1})^{2}ds\int_{t_{m-1}}^{t_{m}}\bigg(\frac{\partial^{2} v^{}(s)}{\partial t^{2}}\bigg)^{2}ds\bigg]dx\nonumber\\
		&&\leq\frac{\Delta t}{3} \int_{t_{m-1}}^{t_{m}}\big\| v_{tt}(s)\big\|^{2}ds.
	\end{eqnarray}
	Since $\|\xi^{*0}(x,y_{j})\|_{0}^{2} = 0$ for all $m$, $1 \leq m \leq M$, summing over $m=1$ to $m$, and using Equation \eqref{E26}-\eqref{E28}, we obtain
	\[\big\|\xi^{*m}(x,y_{j})\big\|_{0}^{2}+ C	\Delta t\sum_{l=1}^{m}\big|\xi^{* l}(x,y_{j})\big\|_{DG}^{2}\leq C\bigg((\Delta t)^{2}\int_{0}^{t_{m}}\big\|v_{tt}\big\|^{2}ds +(\mathcal{N}^{-1}\ln \mathcal{N})^{2k} \bigg).\]
	Similarly, we obtain
	\[\big\|\xi^{\dagger m}(x_{i},y)\big\|_{0}^{2}+ C	\Delta t\sum_{l=1}^{m}\big\|\xi^{\dagger l}(x_{i},y)\big\|_{DG}^{2}\leq C\bigg((\Delta t)^{2}\int_{0}^{t_{m}}\big\|w_{tt}\big\|^{2}ds +(\mathcal{N}^{-1}\ln \mathcal{N})^{2k} \bigg).\]
	Thus, we conclude the proof.
\end{proof}
\begin{theorem}\label{theorem4}
	Let $Pv$ and $Pw$ be the $L^{2}$-projection of the solution $v$ and $w$, respectively, and $v^{m}_{h}(x,y_{j})$ and $w^{m}_{h}(x_{i},y)$ be the numerical solution of the Equation \eqref{E10} and \eqref{E12}, respectively. Assume that $v\in H^{k+1}(I_x)$ and $w\in H^{k+1}(I_y).$  Then we have
	\begin{eqnarray}
		\big\|\xi^{*m}(x,y_{j})\big\|_{DG}^{2}&\leq& C\bigg((\Delta t)^{2}\int_{0}^{t_{m}}\big\|v_{tt}(s)\big\|^{2}ds +(\mathcal{N}^{-1}\ln \mathcal{N})^{2k} \bigg),\nonumber\\
		\big\|\xi^{\dagger m}(x_{i},y)\big\|_{DG}^{2}&\leq& C\bigg((\Delta t)^{2}\int_{0}^{t_{m}}\big\|w_{tt}(s)\big\|^{2}ds +(\mathcal{N}^{-1}\ln \mathcal{N})^{2k} \bigg).\nonumber
	\end{eqnarray}
\end{theorem}
\begin{proof} Taking $\phi=\partial_{\theta}\xi^{*m}$, in Equation \eqref{E25}, we obtain
	\[(\partial_{\theta}\xi^{*m},\partial_{\theta}\xi^{*m})+B_{x}(\xi^{*m},\partial_{\theta}\xi^{*m})= \bigg(\partial_{\theta}Pv^{m}-\frac{\partial v^{m}}{\partial t},\partial_{\theta}\xi^{*m}\bigg)-B_{x}(\rho^{* m},\partial_{\theta}\xi^{*m}).\]
	It follows from  the Cauchy-Schwarz inequality; we obtain
	\begin{eqnarray*}
		&\Delta t\|\partial_{\theta}\xi^{*m}\|_{0}^{2}+B_{x}(\xi^{*m},\xi^{*m})-B_{x}(\xi^{*m},\xi^{* m-1})\nonumber\\
		&\leq\frac{\Delta t}{2}\bigg\|\partial_{\theta}Pv^{m}-\frac{\partial v^{m}}{\partial t}\bigg\|_{0}^{2}+\frac{\Delta t}{2}\|\partial_{\theta}\xi^{*m}\|_{0}^{2}-\Delta tB_{x}(\rho^{* m},\partial_{\theta}\xi^{*m}).
	\end{eqnarray*}
	From Equation \eqref{E26}, \eqref{E27} and \eqref{E28}, for any $m$-th time-step $\forall$ $1\leq m\leq M$. summing $m=1$ to $m$,  we get
	\begin{eqnarray*}
		&&\frac{\Delta t}{2}\sum_{l=1}^{m}\big\|\partial_{\theta}\xi^{l}(x,y_{j})\big\|_{0}^{2}+	\frac{1}{2}\big\|\xi^{*m}(x,y_{j})\big\|_{DG}^{2}\nonumber\\
		&&\leq C\bigg((\Delta t)^{2}\int_{0}^{t_{m}}\big\|v_{tt}\big\|^{2}ds +(\mathcal{N}^{-1}\ln \mathcal{N})^{2k} \bigg).
	\end{eqnarray*}
	Thus, we conclude the proof. 
\end{proof}
\begin{theorem}\label{theorem5}
	Let $v$ and $w$ denote the exact solutions corresponding to the subproblems \eqref{E3} and \eqref{E4}, respectively, and define $u = v + w$. 
	Let $v^{m}_{h}(x,y_{j})$ and $w^{m}_{h}(x_{i},y)$ be the numerical solutions of \eqref{E10} and \eqref{E12}, respectively, and define
	$u^{m}_{h} = v^{m}_{h}(x,y_{j}) + w^{m}_{h}(x_{i},y)$. Then we have the following error estimate:
	\begin{eqnarray}
		\big\|\mathcal{E}^{m}_{h,i,j}\big\|_{0}^{2}+ C	\Delta t\sum_{l=1}^{m}\big\|\mathcal{E}^{l}_{h,i,j}\big\|_{DG}^{2}&\leq& C\bigg((\Delta t)^{2}\int_{0}^{t_{m}}\big\|u_{tt}\big\|^{2}ds +(\mathcal{N}^{-1}\ln \mathcal{N})^{2k} \bigg),\nonumber\\
		\big\|\mathcal{E}^{m}_{h,i,j}\big\|_{DG}^{2}&\leq& C\bigg((\Delta t)^{2}\int_{0}^{t_{m}}\big\|u_{tt}\big\|^{2}ds +(\mathcal{N}^{-1}\ln \mathcal{N})^{2k} \bigg),\nonumber
	\end{eqnarray}
	where \[\big\|\mathcal{E}^{m}_{h,i,j}\big\|_{DG}\leq \big\|E^{* m}_{h,j}\big\|_{DG}+\big\|E^{\dagger m}_{h,i}\big\|_{DG}.\]
\end{theorem}
\begin{proof} 
	The desired result follows from Theorems  \eqref{theorem1}, \eqref{theorem3}, and \eqref{theorem4}.
\end{proof}
\section{Numerical experiments}\label{sec(2.5)}
In this section, we conduct numerical experiments to validate the theoretical results established in the preceding section. Two numerical examples are considered: one involving constant coefficients and another involving variable coefficients.  The outcomes are illustrated through figures and tables, enabling a direct comparison between the theoretical predictions and the computed solutions. These visualizations provide strong evidence for the accuracy, robustness, and reliability of the proposed method. For the numerical simulations, we consider two discretization strategies. Specifically, for $k = 1$, we set $\Delta t = \mathcal{N}^{-1}$, and for $k = 2$, we set $\Delta t = \mathcal{N}^{-2}$, ensuring sufficient temporal resolution relative to the spatial discretization.
\begin{example}\label{example1}
	Consider the following two-dimensional one-parameter singularly perturbed IBVPs on $\mathcal{D} = I$$\times$$(0, T]:$
	\begin{equation*}
		\left\{
		\begin{array} {ll}
			\frac{\partial u}{\partial t}-\mu \Delta u + u_x +u_y = r(x,y,t), & (x,y,t) \in \mathcal{D},\\[8pt]
			u(x,y,0)=f(x,y), &  (x,y) \in \overline{I},\\[8pt]
			u(x,y,t)=0, & (x,y,t) \in \partial I\times(0,T].
		\end{array}\right.
	\end{equation*}
\end{example}
\begin{example}\label{example2}
	Consider the following  two-dimensional one-parameter singularly perturbed IBVPs on $\mathcal{D} = I$$\times$$(0, T]:$
	\begin{equation*}
		\left\{
		\begin{array} {ll}
			\frac{\partial u}{\partial t}-\mu \Delta u + (1+x(1-x))u_x +(1+y(1-y))u_y + u = r, \quad (x,y,t) \in \mathcal{D},\\[8pt]
			u(x,y,0)=f(x,y), \quad \quad  (x,y) \in \overline{I},\\[8pt]
			u(x,y,t)=0, \qquad \qquad \quad (x,y,t) \in \partial I\times(0,T].
		\end{array}\right.
	\end{equation*}
\end{example}
For both Examples \ref{example1} and \ref{example2}, the source term $r(x,y,t)$ 
and initial condition $f(x,y)$ is constructed from the same exact solution
\[u(x,y,t) = (1 - e^{-t})\big(C_{0} + C_{1}x + exp(-(1-x)/\mu)\big)\big(C_{0} + C_{1}y + exp(-(1-y)/\mu)\big),\]
where
\[
C_{0} = -exp(-1/\mu), \qquad C_{1} = -1 - C_{0}.
\]
The decomposition is chosen as 
\begin{equation*}
	r(x,y,t) =	r_{1}(x,y,t) + r_{2}(x,y,t),
\end{equation*}
where
\begin{equation*}
	r_{2}(x,y,t) =	r(x,0,t) + y(r(x,1,t)-r(x,0,t)).
\end{equation*}
In obtaining the numerical solution through the DDG method, we take $\eta_{0}$ from Table~\ref{table_1}, and set $\eta_{1}=\frac{1}{2k^{2}+2k}$, as given in {\cite[Lemma $3.2$]{Cao_DDG}}. Furthermore, in \eqref{flux3} and \eqref{flux4}, we set $\theta = \frac{1}{2}$. The numerical convergence rate is defined by $\mathcal{R} = \log_{2}\left(\frac{e^{\mathcal{N}\times \mathcal{N},M }}{e^{2\mathcal{N} \times 2\mathcal{N},2M}}\right),$ where $e^{\mathcal{N} \times \mathcal{N},M}=\bigg(\sum_{l=1}^{M}\Delta t\|u^{l}-u^{l}_{h}\|_{DG}\bigg)^{1/2}$ denotes the computational error obtained using $\mathcal{N}\times \mathcal{N}$ spatial intervals and time step $\Delta t$.

To assess the accuracy and convergence properties of the proposed method, we evaluate both the DG-norm error and the corresponding convergence rate. The numerical results for Example~\ref{example1} and \ref{example2}  are presented in Table~\ref{tbl_2} and ~\ref{tbl_3}, respectively. These tables report the numerical performance of the proposed method for polynomial degrees $k=1$ and $k=2$ under various values of the singular perturbation parameter $\mu$.

To further demonstrate the capability of the method in resolving boundary layers, surface plots of the numerical and exact solutions are provided for the parabolic IBVP. Specifically, for $\mu = 10^{-4}$, computation are performed using $\mathcal{N} = 32$ and $\Delta t = \frac{1}{32}$ for Example~\ref{example1}. The corresponding results are illustrated in Figure~\ref{figure1}.
\begin{table}[htbp]
	\caption{\it{Errors and observed convergence orders for Example \ref{example1} with $k=1,$ we set $\Delta t = \mathcal{N}^{-1}$  and $k=2,$ we set $\Delta t = \mathcal{N}^{-2}$.}}\label{tbl_2}
	\begin{tabular}{@{}llllll@{}}
		\hline
		&\multicolumn 1 {c}{$\mu$}  & \multicolumn 4 {c}
		{$\mathcal{N} \times \mathcal{N}$}\\
		\hline
		&	  & $8\times 8$ & $16\times 16$ & $32\times 32$  & $64\times 64$  \\
		\hline
		
		&	$10^{-4}$ &        6.4096e-02 &  3.3772e-02 &  1.7497e-02  & 8.8993e-03  \\
		&	  &  0.924  &  0.948  &  0.975 &    \\ [6pt]
		
		${k=1}$	&	$10^{-6}$  &      6.4479e-02  & 3.3784e-02  & 1.7502e-02 &  8.9021e-03\\
		&	          &   0.932 &   0.948  &  0.975 &   \\ [6pt]
		
		&	$10^{-8}$    &    6.4589e-02 &  3.3812e-02 &  1.7506e-02 &  8.9027e-03 \\
		&	          &  0.933  &  0.949  &  0.975 & \\ [6pt]
		
		\hline
		
		&	$10^{-4}$  &   3.4587e-04  & 9.5899e-05   & 2.5862e-05  & 6.8237e-06 \\
	&     &	   1.850   & 1.890  &  1.922  &\\ [6pt]
	
		${k=2}$&	$10^{-6}$  &   2.9976e-04  & 8.7210e-05   & 2.4913e-05  & 6.6921e-06 \\
	&     &	   1.781  & 1.807  &  1.896  &\\ [6pt]
	
	&	$10^{-8}$   &   2.9975e-04  & 8.7206e-05   & 2.4911e-05  & 6.6921e-06 \\
	&     &	   1.781  & 1.807  &  1.896  &\\ [6pt]
		\hline
	\end{tabular}
\end{table}
\begin{figure}[htbp]
	\centering
	\includegraphics[width=0.85\textwidth]{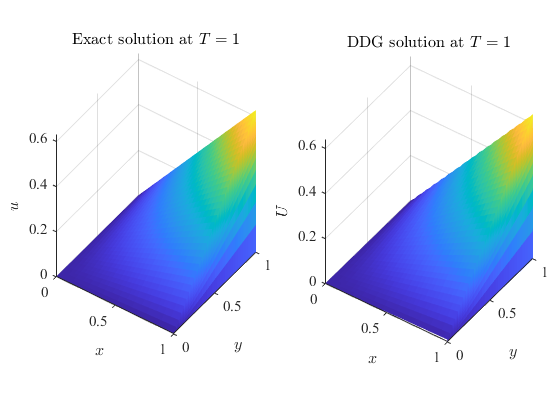}
	\caption{Comparison of the exact and DDG solutions for Example~\ref{example1} at $T=1$, obtained on a $32 \times 32$ mesh with $\Delta t = \frac{1}{32}$, for $\mu=10^{-4}$.}\label{figure1}
\end{figure}

\begin{table}[htbp]
	\caption{\it{Errors and observed convergence orders for Example \ref{example2} with $k=1,$ we set $\Delta t = \mathcal{N}^{-1}$  and $k=2,$ we set $\Delta t = \mathcal{N}^{-2}$.}}\label{tbl_3}
	\begin{tabular}{@{}llllll@{}}
		\hline
		&\multicolumn 1 {c}{$\mu$}  & \multicolumn 4 {c}
		{$\mathcal{N} \times \mathcal{N}$}\\
		\hline
		&	   & $8\times 8$ & $16\times 16$ & $32\times 32$ & $64\times 64$  \\
		\hline
		
		&	$10^{-4}$ &    8.4767e-02  & 3.4351e-02 &  1.7282e-02 &  8.8125e-03  \\
		&	  &   1.303 &   0.991 &   0.971 &    \\ [6pt]
		
		${k=}1$	&	$10^{-6}$  &     6.6851e-02 &  3.3943e-02 &  1.7159e-02 &  8.7591e-03 \\ 
		&	          & 0.977 &   0.984 &   0.970 &   \\ [6pt]
		
		&	$10^{-8}$    & 7.3029e-02  & 3.5062e-02  & 1.7592e-02 &  8.7766e-03 \\
		&	          &  1.058  &  0.995  &  1.003 & \\ [6pt]
		
		\hline
		
		&	$10^{-4}$  &    2.5427e-05 &  5.9351e-06 &  1.4671e-06 &  3.7285e-07 \\
		&     &	   2.099 &  2.016 &    1.976   &\\ [6pt]
		
		${k=}2$&	$10^{-6}$  &      2.5393e-06 &  5.9256e-07 &  1.4640e-07 &  3.7196e-08 \\
		&	 &   2.099 &  2.017  &  1.976 &  \\ [6pt]
		
		&	$10^{-8}$    &      2.5393e-06 &  5.9255e-07  & 1.4640e-07  & 3.7195e-08  \\
		&	          &     2.099  &  2.017  &  1.976 & \\ [6pt]	
		\hline
	\end{tabular}
\end{table}
\section{Conclusion}\label{Conclusion}
In this work, we developed an ADI-type operator-splitting DDG method for two-dimensional one-parameter singularly perturbed parabolic problems, thereby addressing a gap in the literature, where the DDG method has not yet been systematically investigated for two-dimensional singularly perturbed  parabolic problems. The proposed scheme decomposes the original two-dimensional problem into a sequence of one-dimensional subproblems, thereby reducing the number of degrees of freedom and substantially lowering the computational cost and memory requirements compared with a direct two-dimensional discretization. The spatial discretization was carried out using the DDG method on a  piecewise-uniform Shishkin mesh, which effectively resolves the boundary layers arising from the singular perturbation, while the backward Euler method was employed for time integration owing to its unconditional stability for SPPs. The coercivity and stability properties of the proposed fully discrete formulation are established, and DG-norm error bounds for the spatial discretization, independent of the perturbation parameter $\mu$, are derived, demonstrating the robustness of the proposed method. Furthermore, we proved the consistency and convergence of the fully discrete scheme. The theoretical analysis yielded an optimal spatial convergence rate of order $k$, where $k$ denotes the degree of the polynomial basis functions, together with first-order convergence in time. Numerical experiments involving both constant and variable convection coefficients confirmed the theoretical error estimates and demonstrated the robustness and efficiency of the proposed method.

\section*{Acknowledgment}
The first author gratefully acknowledges the financial support provided by the National Institute of Technology, Tiruchirappalli, for his research.

\section*{Conflict of interest}
The authors declare that they have no conflicts of interest.

\section*{Author contribution}
Vijay Kumar: Writing – original draft, Formal analysis, Methodology, Numerical computations.
Gautam Singh: Writing – review and editing, Investigation, Conceptualization, Numerical computations.

\end{document}